\documentclass[reqno, 12pt]{amsart}

\usepackage[utf8]{inputenc}
\usepackage[T1]{fontenc}
\usepackage{lmodern}
\usepackage{mathrsfs} %script-like math font

\usepackage{amsmath}
\usepackage{amsthm}
\usepackage{amssymb}
\usepackage{mathscinet} %pour symboles spéciaux de biblio
\usepackage[all,cmtip]{xy} %pour les diagrammes
\usepackage{stmaryrd} % crochets nombres entiers

\usepackage{enumerate} %pour (i) (a)...
\usepackage{graphicx}
\usepackage{hyperref}
\usepackage[a4paper, hmargin=3cm, vmargin=3cm]{geometry}

\newcommand\mathens[1]{\mathbb{#1}} %fonte des ensembles classiques
\newcommand{\ud}{\mathrm{d}}

\newcommand{\F}{\mathens{F}}
\newcommand{\N}{\mathens{N}}
\newcommand{\Z}{\mathens{Z}}

\newcommand{\R}{\mathens{R}}
\newcommand{\RP}{\mathens{R}\mathrm{P}}
\newcommand{\C}{\mathens{C}}
\newcommand{\CP}{\mathens{C}\mathrm{P}}

\newcommand\sphere[1]{\mathens{S}^{#1}}

\newcommand{\bock}{\mathcal{B}}

\newcommand{\la}{\left\langle}
\newcommand{\ra}{\right\rangle}

\newcommand{\ham}{\textup{Ham}}

\newcommand{\cont}{\textup{Cont}}
\newcommand{\contoc}{\textup{Cont}_{0}}

\newcommand{\id}{\textup{id}}

\DeclareMathOperator\ind{ind}

\DeclareMathOperator\cat{cat}
\DeclareMathOperator\cwgt{cwgt}

\newtheorem{thm}{Theorem}[section]
\newtheorem{lem}[thm]{Lemma}
\newtheorem{cor}[thm]{Corollary}
\newtheorem{prop}[thm]{Proposition}

\newtheorem{prop-def}[thm]{Definition-proposition}

\theoremstyle{definition}

\theoremstyle{remark}
\newtheorem{rem}[thm]{Remark}

\begin{document}

\title{On the minimal number of translated points
in contact lens spaces}

\author[S. Allais]{Simon Allais}
\address{Simon Allais,
\'Ecole Normale Sup\'erieure de Lyon,
UMPA\newline\indent  46 all\'ee d'Italie,
69364 Lyon Cedex 07, France}
\email{simon.allais@ens-lyon.fr}
\urladdr{http://perso.ens-lyon.fr/simon.allais/}
\date{March 30, 2021}

\begin{abstract}
    In this article, we prove that
    every contactomorphism of any standard contact lens space of dimension
    $2n-1$ that is contact-isotopic to identity has at least
    $2n$ translated points.
    This sharp lower bound refines a result of
    Granja-Karshon-Pabiniak-Sandon and answers a conjecture
    of Sandon positively.
\end{abstract}
\maketitle

\section{Introduction}

Given a contact manifold $(V,\ker\alpha)$, a discriminant point of
a contactomorphism $\varphi\in\cont(V,\ker\alpha)$ is a point $p\in V$
such that $\varphi(p)=p$ and $(\varphi^*\alpha)_p = \alpha_p$
(it does not depend on the choice of contact form).
Generically, contactomorphisms do not possess any discriminant points.
A translated point of $\varphi$ with time-shift $\tau$ is a discriminant point 
of $\phi^\alpha_{-\tau} \circ\varphi$,
where $(\phi^\alpha_t)$ denotes the Reeb flow of $\alpha$ ;
it does depend on the choice of contact form $\alpha$.
Contactomorphisms of a closed manifold $V$
that are $C^1$-close to identity have at least as many translated points as some smooth
maps $V\to\R$ (see below).

Inspired by the Arnol'd conjecture, Sandon asked whether any contactomorphism
contact-isotopic to identity $\varphi\in\contoc(V)$ has at least as many
translated points as the minimal number of critical points a smooth map $V\to\R$
can have.
This question was positively answered by Sandon for contact spheres
$\sphere{2n-1}$ and real projective spaces $\RP^{2n-1}$ endowed with the
standard contact form
\begin{equation}\label{eq:stdContactForm}
    \alpha := \frac{1}{2\pi} \sum_{j=1}^n (x_j\ud y_j - y_j \ud x_j),
\end{equation}
(the normalization is made so that the Reeb flow 
$\phi^\alpha_t:z\mapsto e^{2i\pi t}z$ is $1$-periodic)
\cite{San13}.
In the case of hypertight contact manifold
(\emph{i.e.} such that all Reeb orbits are non-contractible for
some contact form supporting the contact structure),
the question was positively answered for generic contactomorphisms contact-isotopic
to identity by the work of Albers-Fuchs-Merry \cite{AFM15} completed by
Meiwes-Naef \cite{MN18}.
The case of the standard lens spaces was partially answered by
Granja-Karshon-Pabiniak-Sandon \cite{GKPS}.
A lens space $L^{2n-1}_k$ with fundamental group $\Z/k\Z$ is
the quotient of $\sphere{2n-1}$ under the free $\Z/k\Z$-action
generated by
\begin{equation}\label{eq:ZpAction}
    (z_1,\ldots,z_n) \mapsto \left(e^{2i\pi w_1/k}z_1,\ldots,e^{2i\pi
    w_n/k}z_n\right),
\end{equation}
where $w_1,\ldots, w_n\in \N^*$ are positive integers prime to $k$.
This definition does depend on the choice of integers $w_1,\ldots,w_n$
but not the results stated here so we will not write them explicitly.
The standard contact form (\ref{eq:stdContactForm}) is invariant under these
$\Z/k\Z$-actions, endowing the associated lens spaces with their standard
contact structure.
In this article, we give a positive answer to the conjecture
of Sandon in the case of standard contact lens spaces.

In order to state the general result of the paper properly,
let us fix a free $\Z/k\Z$-action (\ref{eq:ZpAction})
and let us denote by $\contoc^{\Z/k\Z}(\sphere{2n-1})$ the set
of $\Z/k\Z$-equivariant contactomorphisms of $\sphere{2n-1}$
that are isotopic to identity through $\Z/k\Z$-equivariant
contactomorphisms.
The set of time-shifts of a contactomorphism $\varphi\in\contoc(\sphere{2n-1})$
will refer to the set of time-shifts of its translated points.
Since the Reeb flow of (\ref{eq:stdContactForm}) is $1$-periodic,
the set of time-shifts of $\varphi$ is invariant under integral translations so it
can be seen as a subset of $\R/\Z$.

\begin{thm}\label{thm:main}
    Let us assume that $\varphi\in\contoc^{\Z/k\Z}(\sphere{2n-1})$ has finitely
    many translated points (with the standard contact form), then the number of
    time-shifts
    of $\varphi$ in $\R/\Z$ is at least $2n$.
\end{thm}

More precise information about the structure of the set of translated
points of $\varphi$ can be also recovered from the proof,
see Remark~\ref{rem:LS}.
Theorem~\ref{thm:main} implies that every contactomorphism
$\varphi\in\contoc(L^{2n-1}_k)$ has at least $2n$ translated points.
Indeed, one can lift $\varphi$ to $\widetilde{\varphi}\in\contoc^{\Z/k\Z}(\sphere{2n-1})$.
Translated points of $\widetilde{\varphi}$ come in families of $\Z/k\Z$-orbits
and translated points of a same orbit have a common time-shift $t\in\R/\Z$.
Since distinct $\Z/k\Z$-orbits of translated points of $\widetilde{\varphi}$
projects to distinct translated points of $\varphi$, the conclusion follows.

\begin{cor}
    Every contactomorphism of $L^{2n-1}_k$ contact-isotopic to identity
    has at least $2n$ translated points.
\end{cor}

More precisely, Theorem~\ref{thm:main} implies that every 
contactomorphism $\varphi_1$ of $L^{2n-1}_k$ contact-isotopic to identity
through $(\varphi_t)$
has at least $2n$ translated points $\{x_j\}$ so that
the loops
\begin{equation*}
    t\mapsto
    \begin{cases}
        \varphi_{2t}(x_j), & t\in [0,1/2], \\
        \phi^\alpha_{(1-2t)\tau_j} \circ \varphi_1(x_j), & t\in [1/2,1],
    \end{cases}
\end{equation*}
are contractible for some choice of respective time-shifts $\{\tau_j\}$.
Moreover, if the set $\{ x_j \}$ is finite, the cardinal of the projection
of $\{ \tau_j \}$ in $\R/\Z$ is at least $2n$
(and more information about the structure of these sets can be recovered
from the proof, see Remark~\ref{rem:LS}).

This lower bound is optimal.
Indeed, let $\mathrm{gr}(\varphi)$ be the graph of a contactomorphism
$\varphi\in \contoc(V,\alpha)$
\begin{equation*}
    \mathrm{gr}(\varphi) := \{ (p,\varphi(p),g(p)) \ |\ p\in V \}
    \subset V\times V\times\R,
\end{equation*}
where $\varphi^*\alpha = e^g\alpha$.
The graph of $\varphi$ is a Legendrian for the
contact form $\alpha_2 - e^\theta \alpha_1$ where $\alpha_j$ is the pull-back
of $\alpha$ by the projection on the $j$-th factor and $\theta$ denotes the
coordinate of the $\R$ factor.
Discriminant points of $\varphi$ correspond to intersection of $\mathrm{gr}(\varphi)$
with $\mathrm{gr}(\id)=\Delta\times 0$ where $\Delta\subset V\times V$ is the diagonal.
If $\varphi$ is $C^1$-close to identity, then $\mathrm{gr}(\varphi)$ lies in a Weinstein
neighborhood $W\simeq J^1\Delta$ of $\Delta\times 0$.
Seen in $J^1\Delta$, the graph of $\varphi$ is the $1$-jet $j^1 f$ of a smooth map
$f:\Delta\to \R$ and translated points with small time-shift $t$ correspond to
critical points of $f$ with critical value $t$, as stated.
In the special case where all Reeb orbits of $(V,\alpha)$ are closed,
one can shrink the Weinstein neighborhood $W$ so that the Reeb orbits
of $\mathrm{gr}(\id)$ correspond \emph{only}
to Reeb orbits of $\Delta\times 0 \subset J^1\Delta$
(whereas if the Reeb orbit $\gamma$ of some point $(x,x,0)\in\mathrm{gr}(\id)$
is not closed, the Poincaré recurrence theorem implies that
$\gamma$ contains a sequence of $(x,y_k,0)$ with $y_k\to x$ and $y_k\neq x$ ;
so that every Weinstein neighborhood must contain some point in the
same Reeb orbit as $(x,x,0)$ but in a different Reeb orbit of the neighborhood
$W\simeq J^1\Delta$).
Therefore, when all Reeb orbits of $(V,\alpha)$ are closed, there exists
contactomorphisms contact-isotopic to identity with exactly as many translated
points as the minimal number of critical point a map $V\to\R$ can have
(this discussion was already present in the work of Sandon \cite{San13}).

Since the Reeb flow of $L^{2n-1}_k$ is periodic,
it is enough to find a map $L^{2n-1}_k\to\R$ with $2n$ critical points.
A classical way of doing so is to consider the quotient map of the
$\Z/k\Z$-invariant map of $\sphere{2n-1}$
\begin{equation*}
    (z_1,\ldots,z_n) \mapsto |z_1|^2 + 2|z_2|^2 + \cdots + n|z_n|^2.
\end{equation*}
This is a Morse-Bott map $L^{2n-1}_k\to\R$ whose critical set is the disjoint
union of $n$ circles.
By properly adding a Morse map of the circle with $2$ critical points
in the Morse-Bott neighborhood
of each critical circle, one obtains a Morse map with $2n$ critical points.

Since every $\Z/k\Z$-equivariant map of $\sphere{2n-1}$ is also equivariant
under a subgroup $\Z/p\Z\subset\Z/k\Z$
where $p$ is a prime dividing $k$, it is enough to prove Theorem~\ref{thm:main}
when $k$ is a prime.
The case when $k$ is even was already proven by Sandon \cite{San13}.
Granja-Karshon-Pabiniak-Sandon proved Theorem~\ref{thm:main} when translated
points of $\varphi$ are non-degenerate and proved that the number of
translated points is at least $n$ otherwise \cite[Corollary~1.3]{GKPS}.

In order to study translated points, we follow the idea of Sandon originally
used in the case of $\sphere{2n-1}$ and $\RP^{2n-1}$ in \cite{San13} and
extended by Granja-Karshon-Pabiniak-Sandon to lens spaces $L_k^{2n-1}$ in \cite{GKPS}.
Sandon discovered that the proofs of the Arnol'd conjecture in $\CP^{n-1}$
given by Givental \cite{Giv90} and Théret \cite{The98} can be extended
to study translated points on quotients of the contact sphere $\sphere{2n-1}$.
In this article we are deeply inspired by the point of view of Théret on
the Fortune-Weinstein theorem \cite{For85} that we
had already develop to study periodic points of Hamiltonian
diffeomorphisms in $\CP^{n-1}$ \cite{periodicCPd, OnHoferZehnderGF} whereas
the study of Granja-Karshon-Pabiniak-Sandon was mainly focused on the
properties of a quasimorphism inspired by Givental's
non-linear Maslov index on $\CP^{n-1}$.

In order to study the $\Z/k\Z$-equivariant
contactomorphism $\varphi\in\contoc^{\Z/k\Z}(\sphere{2n-1})$ of the sphere $\sphere{2n-1}\subset\C^n$,
we lift it to a Hamiltonian diffeomorphism $\Phi$ of 
the symplectization $\C^n\setminus 0$ (endowed with the standard symplectic
form $\ud x\wedge\ud y$)
\begin{equation}\label{eq:lift}
    \Phi(z) := \frac{\|z\|}{e^{\frac{1}{2}g\left(\frac{z}{\|z\|}\right)}}
    \varphi\left(\frac{z}{\|z\|}\right),\quad
    \forall z\in\C^n,
\end{equation}
where $\varphi^*\alpha = e^g\alpha$ and $\|\cdot\|$ denotes the usual
Euclidean norm.
By construction, $\Phi$ is $\R_+$-homogeneous and equivariant under the
$\Z/k\Z$-action extended by linearity from the action on $\sphere{2n-1}$ :
$\Phi$ is an $(\R_+\times\Z/k\Z)$-equivariant Hamiltonian diffeomorphism.
$\Z/k\Z$-orbits of translated points with time-shift $t$ then correspond to
$(\R_+\times\Z/k\Z)$-orbits of fixed points of $e^{-2i\pi t}\Phi$.
For such families of Hamiltonian diffeomorphisms $(e^{-2i\pi t}\Phi)$,
Givental and Théret developed a variational theory based on generating
functions (they restricted themselves to homogeneous Hamiltonians
invariant under $S^1$ or $\Z/2\Z$, the extension to $\Z/k\Z$
was considered by Granja-Karshon-Pabiniak-Sandon).
Theorem~\ref{thm:main} follows from a careful application of the
Lusternik-Schnirelmann theory to this variational setting.

\subsection*{Organization of the paper}
In Section~\ref{se:LS}, we recall results of the Lusternik-Schnirelmann
theory that are useful for us.
In Section~\ref{se:optimal}, we introduce the variational setting and
use it to prove Theorem~\ref{thm:main}.

\subsection*{Acknowledgment}
I am grateful to Yael Karshon and Margherita Sandon for their advices and support.
I also thank Daniel \'{A}lvarez-Gavela, Mohammed Abouzaid and the other organizers
of the \emph{Generating Function Day} of February 2021, where I was able to
discuss this subject with
Yael for the first time.
Finally, I would like to thank my advisor Marco Mazzucchelli for his
constant support and helpful advice.

\section{Lusternik-Schnirelmann theory on the relative case}
\label{se:LS}

\subsection{Lusternik-Schnirelmann theorem}

Let $X$ be an ANR topological space (metric is enough for us).
Given $A\subset X$, the integer $\cat_X A\in\N$ is the minimal
integer $k$ such that $A$ can be covered by $k$ open subsets that are
contractible in $X$.
The Lusternik-Schnirelmann category of $X$ is defined as
\begin{equation*}
    \cat X := \cat_X X.
\end{equation*}

Let us fix $Y\subset X$. Given a subset $A\subset X$ that contains $Y$,
the integer $\cat_X (A,Y) \in\N$ is the minimal integer $k$ such that
\begin{equation*}
    A \subset W\cup A_1 \cup A_2 \cup \cdots \cup A_k,
\end{equation*}
with $W\subset X$ open that contains $Y$ and retracts on $Y$ by deformation and
open subsets
$A_1,\ldots, A_k$
that are contractible in $X$.
The Lusternik-Schnirelmann category of $(X,Y)$ is defined as
\begin{equation*}
    \cat (X,Y) := \cat_X (X,Y).
\end{equation*}

Let $f:M\to\R$ be a $C^1$-map defined on a closed manifold $M$.
More generally, the following theorem is true for $f$ defined
on a Banach manifold $M$ and that satisfies the Palais-Smale
condition in $f^{-1}(a,b)$.

\begin{thm}[{See for instance \cite[Proposition~7.7]{CLOT}}]
    \label{thm:LS}
    Let $a\leq b$ be regular values of $f$.
    Let us assume that $\{ f\leq b\}$ is connected.
    If $f$ has finitely
    many critical points with value inside $(a,b)$,
    the number of critical values $N\in\N$ inside $(a,b)$ satisfies
    \begin{equation*}
        N \geq \cat (\{ f \leq b \}, \{ f\leq a\}).
    \end{equation*}
\end{thm}

\begin{rem}\label{rem:LS}
Actually, the Lusternik-Schnirelmann theorem tells us a bit more
about the structure of the critical set $K\subset f^{-1}(a,b)$ of $f$.
For instance, if $f$ has less critical values than the category
of its relative sublevel set, non only is the set of critical points $K$
infinite, but any neighborhood $V\subset\{ f\leq b\}$ of $K$ has a category
$\cat_X V \geq 1$ relative to $X:=\{ f\leq b\}$
(in particular, $V$ possess a non-contractible connected component).
The proof of Theorem~\ref{thm:main} is an application of this theorem
and therefore tells us a bit more than stated too.
In particular, if $\varphi\in\contoc^{\Z/k\Z}(\sphere{2n-1})$ has less time-shifts
in $\R/\Z$ than $2n$, for any $\Z/k\Z$-invariant neighborhood $V\subset \sphere{2n-1}$,
the inclusion $V/(\Z/k\Z) \hookrightarrow L_k^{2n-1}$ is not null-homotopic.
\end{rem}

\subsection{Fadell-Husseini's category weight}

In the study of critical points of the lens spaces,
the usual bound $\cat X \geq CL(X) + 1$, where $CL(X)\in\N$ denotes
the cup-length of $X$, is not optimal.
Given $u\in H^*(X)$ and $A\subset X$, the class $u|_A \in H^*(A)$
denotes the image of $u$ under the restriction morphism induced by
the inclusion $A\subset X$.
Given $u\in H^*(X,Y)$ and $(A,B)\subset (X,Y)$, we denote
by $u|_{(A,B)} \in H^*(A,B)$ the restricted class.

Given a non-zero class $u\in H^*(X)$, Fadell-Husseini 
\cite{FH92} defined the category
weight $\cwgt(u)$ of $u$
as the maximal integer $k\in\N$ such that $u|_A = 0$ for all $A\subset X$
satisfying $\cat_X A \leq k$.
It can be naturally extended to non-zero relative classes $u\in H^*(X,Y)$
as the maximal integer $k\in\N$ such that $u|_{(A,Y)}=0$ for all
$A\supset Y$ in $X$ satisfying $\cat_X (A,Y) \leq k$.
For instance, if $X$ is connected, the class $1\in H^0(X)$
has category weight $\cwgt(1)=0$ whereas $\cwgt(u)\geq 1$ for
a non-zero class $u\in H^d(X)$ with degree $d\geq 1$.
In the case of a non-zero relative class $u\in H^*(X,Y)$,
the category weight may vanish.

By the usual properties of the cup-product,
for non-zero classes $u$ and $v$ in $H^*(X)$
or $H^*(X,Y)$, if $u\smile v \neq 0$, then
\begin{equation*}
    \cwgt(u\smile v) \geq \cwgt(u) + \cwgt(v).
\end{equation*}
Indeed, if $u|_A = 0$ and $v|_B = 0$ then $(u\smile v)|_{A\cup B} = 0$.

By definition of the category weight, $\cat (X,Y) \geq 1 + \cwgt (u)$
for all non-zero class $u\in H^*(X,Y)$.
Therefore, if $u_1\smile \cdots \smile u_k$ is a non-zero class of
$H^*(X,Y)$, then
\begin{equation*}
    \cat (X,Y) \geq 1 + \cwgt(u_1) + \cdots + \cwgt(u_k).
\end{equation*}
In principle, it generalises the cup-length estimate.
The following theorem shows that it is indeed a better estimate in general.
Let us denote the mod $p$ Bockstein morphism of a topological space $X$
by $\bock:H^*(X;\F_p) \to H^{*+1}(X;\F_p)$.

\begin{thm}[{\cite[Theorem~1.2]{FH92}}]
    \label{thm:bockstein}
    Let $X$ be an ANR topological space and $p$ be a prime number.
    If $u\in H^1(X;\F_p)$ and $\bock u \neq 0$, then
    $\cwgt(\bock u) \geq 2$.
\end{thm}

As an application, one can compute the Lusternik-Schnirelmann category of
a lens space $L^{2n-1}_p$ of dimension $2n-1$ and fundamental group $\Z/p\Z$.
As a graded $\F_p$-algebra,
\begin{equation*}
    H^*(L_p^{2n-1};\F_p) \simeq \F_p[\alpha,\beta]/(\alpha)(\beta^n),
\end{equation*}
where $\deg \alpha =1$, $\deg \beta = 2$ and $\beta = \bock\alpha$.
The category weight $\cwgt(\alpha)\geq 1$ whereas $\cwgt(\beta) \geq 2$
according to Theorem~\ref{thm:bockstein}. Therefore
\begin{equation*}
    \cat (L_p^{2n-1}) \geq 1 + \cwgt(\alpha\smile\beta^{n-1})
    \geq 1 + \cwgt(\alpha) + (n-1)\cwgt(\beta) \geq 2n.
\end{equation*}
As there exist smooth maps $L^{2n-1}_p \to \R$ with exactly $2n$
critical points, Theorem~\ref{thm:LS} implies that $\cat(L_p^{2n-1})=2n$.
More precisely, Theorem~\ref{thm:LS} implies that every smooth map
$L^{2n-1}_p \to \R$ with finitely many critical points has at least
$2n$ critical values.
The computation of $\cat(L^{2n-1}_p)$ is initially due to Krasnosel'ski\u{\i}
\cite{Kra55}.

\section{The optimal bound on translated points in lens spaces}
\label{se:optimal}

In Section~\ref{se:variational}, we recall the variational principle for
translated points.
As we said in the introduction, it goes back to Givental \cite{Giv90} and
Théret \cite{The98} to study Hamiltonian dynamics of $\CP^{n-1}$
and was then extended by Sandon \cite{San13} and
Granja-Karshon-Pabiniak-Sandon \cite{GKPS} to treat translated points
of quotients of $\sphere{2n-1}$.
We do not strictly follow the lines of \cite{GKPS} but rather
the point of view developed in \cite{periodicCPd,OnHoferZehnderGF},
inspired by ideas of Théret.

\subsection{A variational principle for translated points}
\label{se:variational}

\newcommand\action{\mathcal{T}}

Let $p$ be a prime number greater than $2$
and let us fix $\varphi\in\contoc^{\Z/p\Z}(\sphere{2n-1})$.
Let $\Phi\in\ham(\C^n\setminus 0)$
be the $(\R_+\times \Z/p\Z)$-equivariant lift
given by (\ref{eq:lift}).
Let us build a $C^1$-family of generating functions
$F_t : (\C^{2n})^k \to \R$, $t\in (-1,3)$, associated with
$e^{-2i\pi t}\Phi$ that is 2-homogeneous and invariant
under the diagonal action of $\Z/p\Z$ on $(\C^{2n})^k$ induced
by the action on $\C^{2n}$.
Let $\sigma_1,\ldots,\sigma_m$ be $(\R_+\times \Z/p\Z)$-equivariant
Hamiltonian diffeomorphisms of $\C^n\setminus 0$
whose restriction to $\sphere{2n-1}$ are $C^1$-close to the identity
such that
$
    \Phi = \sigma_m \circ \cdots \circ \sigma_2 \circ \sigma_1,
$
with $m$ even.
Since the $\sigma_j$'s are $C^1$-close to identity among
$(\R_+\times \Z/p\Z)$-equivariant maps, there exist
smooth maps $f_j : \C^n\setminus 0 \to \R$ such that
\begin{equation*}\label{eq:egf}
    \forall z_j\in\C^n, \exists ! w_j\in\C^n,\quad
    w_j=\frac{z_j+\sigma_j(z_j)}{2} \quad \text{ and } \quad
    \nabla f_j (w_j) = i(z_j-\sigma_j(z_j)),
\end{equation*}
(the maps $G_j:=\nabla f_j$ can be found by the implicit function
theorem and these $G_j$ are gradients because their graphs are Lagrangians).
By equivariance of the $\sigma_j$'s, one can choose the $f_j$'s to be
2-homogeneous and $\Z/p\Z$-invariant.
These maps extend to $C^{1,1}$-maps of $\C^n$ by imposing $f_j(0)=0$.
In the special case of the small equivariant Hamiltonian diffeomorphism
$\delta_t:z\mapsto e^{-2i\pi t}z$, the associated map $q_t:\C^n\to\R$
is
\begin{equation*}
    q_t(w) := -\tan (\pi t)\| w\|^2, \quad
    \forall w\in\C^n,
\end{equation*}
well defined for $t\in (-1/2,1/2)$
($\|\cdot\|$ denotes the usual Euclidean norm of $\C^n\simeq\R^{2n}$).
We can now define, for $v_1,\ldots, v_{m+7}\in\C^n$ and $t\in (-1,3)$,
\begin{equation}\label{eq:gf}
    F_t (v_1,\dotsc,v_{m+7}) := \sum_{j=1}^m
        f_j\left(\frac{v_j + v_{j+1}}{2}\right) +
        \sum_{j=m+1}^{m+7} q_{t/7}\left(\frac{v_j + v_{j+1}}{2}\right)
        + \sum_{j=1}^{m+7}
    \frac{1}{2}\la v_j,iv_{j+1}\ra.
\end{equation}
Therefore, the map $(t,\mathbf{v})\mapsto F_t(\mathbf{v})$ is a $C^1$-map
$(-1,3)\times (\C^n)^{m+7}\to\R$ that is
2-homogeneous and $(\R_+\times \Z/p\Z)$-equivariant.
\begin{prop}[{\cite[Proposition~5.1]{periodicCPd}}]
    Critical points $(v_1,\ldots,v_{m+7})$ of $F_t$ are in one-to-one
    correspondence with fixed points of $e^{-2i\pi t}\Phi$ through
    the map 
    \begin{equation*}(v_1,\ldots,v_{m+7})\mapsto v_1.\end{equation*}
\end{prop}
Since $F_t$ is 2-homogeneous and $\Z/p\Z$-invariant, an
$(\R_+\times\Z/p\Z)$-orbit of fixed points of $e^{-2i\pi t}\Phi$ corresponds to
an $(\R_+\times\Z/p\Z)$-orbit of critical points of $F_t$ with critical value
$0$.  Let $L^{2N-1}_p$ be the lens space of dimension $2N-1$ with $N:=n(m+7)$
obtained by quotient of $\sphere{2N-1}$ under the free diagonal action of
$\Z/p\Z$.  Let $\widehat{F}_t$ be the $C^1$-map $L^{2N-1}_p \to \R$ induced by
the $\Z/p\Z$-invariant map $F_t|_{\sphere{2N-1}}$.  Now, an
$(\R_+\times\Z/p\Z)$-orbit of fixed points of $e^{-2i\pi t}\Phi$ corresponds to
a critical point of $\widehat{F}_t$ \emph{with critical value $0$} and this
correspondence is a bijection.

Going back to our initial problem, $\Z/p\Z$-orbits of translated points of
$\varphi$ with
time-shift $t\in (-1,3)$ are in one-to-one correspondence with
$(\R_+\times\Z/p\Z)$-orbits of fixed points of $e^{-2i\pi t}\Phi$.  Therefore,
the $\Z/p\Z$-orbits of 
translated points of $\varphi$ with time-shift $t\in (-1,3)$ are in
one-to-one correspondence with critical points of $\widehat{F}_t$ with critical
value $0$.

According to \cite[Lemma~5.6]{periodicCPd}, the $C^1$-map $(t,z)\mapsto \widehat{F}_t(z)$
is a submersion.
Let $M \subset (-1,3)\times L^{2N-1}_p$ be the $C^1$-submanifold
\begin{equation*}
    M := \left\{ (t,z) \ |\ \widehat{F}_t(z) = 0 \right\}
\end{equation*}
and $\action : M\to (-1,3)$ be the $C^1$-map induced by the projection on
the first factor.
The $\Z/p\Z$-orbits of translated points of $\varphi$ with time-shift $t\in (-1,3)$
are in one-to-one correspondence with critical points of $\action$ with
critical value $t$.
By definition (\ref{eq:gf}), $\widehat{F}_t$ is decreasing with $t$
so $\{ \widehat{F}_t \leq 0 \}$ is increasing.
We recall that we have an isomorphism of graded $H^*(L^{2N-1}_p)$-algebra
\begin{equation}
    \label{eq:isoAction}
    H^*(\{ \action \leq b \},\{ \action \leq a\})
    \simeq
    H^*\left(\left\{ \widehat{F}_b \leq 0 \right\}, \left\{ \widehat{F}_a \leq
    0\right\}\right),
\end{equation}
for $a\leq b$, where $H^*(L^{2N-1}_p)\simeq H^*((-1,3)\times L^{2N-1}_p)$ acts by
cup-product. This isomorphism being induced by a continuous map, it also
commutes with the Bockstein morphism $\bock$.
See for instance \cite[Section~5.4]{periodicCPd}.

Given $A\subset L_p^{2N-1}$, the cohomological index of $A$ is defined
as the rank of the induced restriction morphism in cohomology
\begin{equation*}
    \ind A := \mathrm{rank} \left(H^*(L_p^{2N-1};\F_p) \to H^*(A;\F_p)\right)
    \in \{ 0,1,\ldots, 2N\}.
\end{equation*}
Equivalently, $\ind A$ is the minimal $k\in\N$ such that
the morphism $H^k(L_p^{2N-1};\F_p)\to H^k(A;\F_p)$ is trivial.
Our definition slightly differs from \cite{GKPS} since we are using
singular cohomology instead of \v{C}ech cohomology:
in general it could be less than their index but it coincides on open sets.
The following lemma is proven in \cite[\S Contact Arnold conjecture]{GKPS}.
We explain how to adapt the proof to our setting

\begin{lem}\label{lem:variationIndex}
    Let us assume that $\varphi$ does not have discriminant points,
    then
    \begin{equation*}
        \ind \left\{ \widehat{F}_2 \leq 0 \right\} - 
        \ind \left\{ \widehat{F}_0 \leq 0 \right\} = 4n.
    \end{equation*}
\end{lem}

More generally the variation of index between the sublevel at time $t$
and time $t+k$ is $2nk$ for $k\in\N$ as long as everything is defined
and $t$ is not a time-shift.
Here, we are considering a variation between time $2$ and $0$
instead of considering it between time $1$ and $0$ due to a technical issue
in the proof of Theorem~\ref{thm:main}.

\begin{proof}
    Since $\varphi$ does not have translated points with time-shift $0$,
    $\widehat{F}_0$ and $\widehat{F}_2$ are submersion.
    Therefore, $\{ \widehat{F}_0 \leq 0 \}$ and $\{ \widehat{F}_2 \leq 0 \}$
    are submanifolds with boundary.
    Taking collar neighborhoods, we deduce that their singular cohomologies
    are naturally isomorphic to their \v{C}ech cohomologies.
    According to \cite[Proposition~5.2 and
    Lemma~5.5]{periodicCPd},
    one can find
    $\Z/p\Z$-equivariant automorphisms $A_0$ and $A_2$ of $\C^N$
    such that $F_j\circ A_j = G\oplus Q_j$ where $Q_j$ are
    non-degenerated quadratic form of $(\C^n)^7$
    that are invariant under the diagonal action of $\Z/p\Z$
    and satisfy
    \begin{equation*}
        i(Q_2) - i(Q_0) = 4n,
    \end{equation*}
    where $i(Q)\in\N$ denotes the index of the quadratic form $Q$
    (the referred proofs do not discuss possible $\Z/p\Z$-symmetries
    but it directly follows from the explicit formulas of $A_j$ and $Q_j$
    that are given).
    Therefore, the $\{ \widehat{Q}_j \leq 0 \}$'s retract on
    lens spaces $\simeq L^{i(Q_j)-1}_p$ by deformation.
    The conclusion is now a direct consequence of
    \cite[Proposition~3.9 (v) and Proposition~3.14]{GKPS}.
\end{proof}

\subsection{Proof of the main theorem}

\begin{proof}[Proof of Theorem~\ref{thm:main}]
    As was already mentioned in the introduction, by taking a subgroup,
    one can assume that $\varphi\in\contoc^{\Z/p\Z}(\sphere{2n-1})$
    with $p$ prime.
    The case $p=2$ was already treated by Sandon \cite{San13} and can easily be recovered
    with small modifications of the following proof.
    Hence, we assume that $p\geq 3$ and take back notation of
    the latter section.

    Without loss of generality, we assume that $\varphi$ does not
    have any discriminant point (\emph{i.e.} translated point with time-shift
    $0$).
    Let $k\in\N$ be the number of time-shifts of $\varphi$ seen in $\R/\Z$.
    The number of time-shifts in $[0,2]$ is thus $2k$.
    According to the correspondence with critical values of $\action$,
    the number of critical values in $[0,2]$ is also $2k$.
    If $\varphi$ has finitely many translated points, $\action$ has finitely
    many critical points.
    According to Theorem~\ref{thm:LS}, if $\varphi$ has finitely many translated points,
    \begin{equation*}
        2k \geq \cat(\{ \action|_C \leq 2 \}, \{ \action|_C \leq 0\}),
    \end{equation*}
    for all connected component $C$ of $\{ \action\leq 2\}$.

    In order to conclude, let us show that this category is at least $4n-1$
    for a good choice of $C$.
    As a graded $\F_p$-algebra,
    \begin{equation*}
        H^*(L_p^{2N-1};\F_p) \simeq \F_p[\alpha,\beta]/(\alpha)(\beta^N),
    \end{equation*}
    where $\deg \alpha =1$, $\deg \beta = 2$ and $\beta = \bock\alpha$.
    According to Lemma~\ref{lem:variationIndex},
    we have the following variation of index
    \begin{equation*}
        \ind \left\{ \widehat{F}_2 \leq 0 \right\} - 
        \ind \left\{ \widehat{F}_0 \leq 0 \right\} = 4n.
    \end{equation*}
    If $\ind \{ \widehat{F}_0 \leq 0 \}$ is even $=2r$, then the classes
    \begin{equation*}
        \beta^r,\alpha\beta^r,\ldots, \beta^{r+2n-1},\alpha\beta^{r+2n-1}
        \in H^*(L^{2N-1}_p;\F_p)
    \end{equation*}
    are in the kernel of the restriction to $\{ \widehat{F}_0\leq 0\}$ but not in
    the one of the restriction to $\{ \widehat{F}_2 \leq 0\}$.
    Therefore, there is a class $u\in H^{2r}(\{ \widehat{F}_2 \leq 0\},\{
    \widehat{F}_0\leq 0\})$
    (associated with $\beta^r$) such that the class
    \begin{equation*}
        u\smile \alpha \beta^{2n-1} \in H^{2r+4n-1}\left(\left\{ \widehat{F}_2 \leq
        0\right\},\left\{ \widehat{F}_0\leq 0\right\}\right)
    \end{equation*}
    is non-zero.
    Hence,
    \begin{equation*}
        \cwgt(u\smile \alpha \beta^{2n}) \geq \cwgt(\alpha)
        + (2n-1)\cwgt(\beta) \geq 4n - 1,
    \end{equation*}
    according to Theorem~\ref{thm:bockstein} (we recall that $\beta = \bock\alpha$).
    The isomorphism of $H^*(L^{2N-1}_p)$-algebra 
    commuting with the Bockstein morphism $\bock$ (\ref{eq:isoAction}) allows us
    to conclude that
    \begin{equation*}
        \cat(\{ \action|_C \leq 2 \}, \{ \action|_C \leq 0\}) \geq
        1 + \cwgt(u'\smile \alpha \beta^{2n}) \geq 4n,
    \end{equation*}
    where $u'$ is the image of $u$ under the isomorphism (\ref{eq:isoAction})
    and $C$ is a connected component on which $u'\smile \alpha\beta^{2n}$
    is non-zero
    (we recall that the cohomology of a disjoint union of sets is
    the direct sum of the cohomologies of these sets).\\
    If $\ind\{ \widehat{F}_0 \leq 0\}$ is odd $=2r+1$, then the classes
    \begin{equation*}
        \alpha\beta^r,\beta^{r+1},\ldots, \alpha\beta^{r+2n-1},\beta^{r+2n}
        \in H^*(L^{2N-1}_p;\F_p)
    \end{equation*}
    are in the kernel of the restriction to $\{ \widehat{F}_0\leq 0\}$ but not in
    the one of the restriction to $\{ \widehat{F}_2 \leq 0\}$.
    Therefore, there is a class $u'\in H^{2r+1}(\{\action\leq 2\},\{\action\leq 0\})$
    (associated with $\alpha\beta^r$) such that the class
    \begin{equation*}
        u'\smile \beta^{2n-1} \in H^{2r+4n-1}(\{\action\leq 2\},\{\action\leq 0\})
    \end{equation*}
    is non-zero.
    Hence,
    \begin{equation*}
        \cat(\{ \action|_C \leq 2 \}, \{ \action|_C \leq 0\}) \geq
        1 + \cwgt(u'\smile \beta^{2n-1}) \geq 1 + 2(2n-1)
        \geq 4n -1,
    \end{equation*}
    where $C$ is a connected component on which $u'\smile \beta^{2n-1}$
    is non-zero.
\end{proof}

\bibliographystyle{amsplain}
\bibliography{biblio} 
\end{document}